\documentclass[a4paper]{amsart}
\usepackage{amssymb}
\usepackage{tikz-cd}
\usepackage[all]{xy}
\SelectTips{cm}{}
\calclayout

\title{On algebras of strongly derived unbounded type}
\author{Chao Zhang}
\address{Chao Zhang\\ Department of Mathematics\\ Guizhou University\\ Guiyang 550025\\ P.R. China.}
\address{Faculty of Mathematics, Bielefeld University, 33501 Bielefeld, Germany}
\email{zhangc@amss.ac.cn}


\newtheorem{lemma}{Lemma}[section]
\newtheorem{proposition}[lemma]{Proposition}
\newtheorem{corollary}[lemma]{Corollary}
\newtheorem{theorem}[lemma]{Theorem}

\theoremstyle{remark}
\newtheorem{remark}[lemma]{Remark}

\theoremstyle{definition}

\newtheorem{definition}[lemma]{Definition}

\numberwithin{equation}{section}

\renewcommand{\mod}{\operatorname{mod}\nolimits}

\newcommand{\Supp}{\operatorname{Supp}\nolimits}

\newcommand{\proj}{\operatorname{proj}\nolimits}
\newcommand{\rad}{\operatorname{rad}\nolimits}

\newcommand{\Mod}{\operatorname{Mod}\nolimits}

\newcommand{\End}{\operatorname{End}\nolimits}

\newcommand{\Hom}{\operatorname{Hom}\nolimits}

\newcommand{\Image}{\operatorname{Im}\nolimits}
\newcommand{\Ker}{\operatorname{Ker}\nolimits}

\renewcommand{\dim}{\operatorname{dim}\nolimits}

\newcommand{\hl}{\mathrm{hl}}
\newcommand{\hw}{\mathrm{hw}}
\newcommand{\hr}{\mathrm{hr}}

\begin{document}

\begin{abstract}
Let $A$ be a finite-dimensional algebra
over an algebraically closed field. We prove
$A$ is a strongly derived unbounded algebra if and
only if there exists an integer $m$, such that $C_m(\proj A)$, the category
of all minimal projective complexes with degree concentrated in $[0, m]$,
is of strongly unbounded type, which is also equivalent to the statement
the repetitive algebra $\hat{A}$ is of
strongly unbounded representation type.
As a corollary, we
can establish the dichotomy on the representation type of $C_m(\proj A)$,
the homotopy category $K^b(\proj A)$ and the repetitive algebra $\hat{A}$.

\end{abstract}


\maketitle

\setcounter{tocdepth}{1}

\section*{Introduction}

Throughout this article, $k$ is an algebraically closed field and all the algebras are
associative finite dimensional connected basic $k$-algebras with identity.
During the research of representation theory of algebras, one of the main topics is to study
the representation type. As early as 1940s, Brauer and Thrall began
the investigation of representation type of finite dimensional algebras \cite{Br41, Th47}.
Jans formulated the first and second Brauer-Thrall conjectures for finite dimensional
algebras in his paper \cite{Jan57}, roughly speaking, the first Brauer-Thrall conjecture
says that an algebra is of bounded representation type if and only if
it is of finite representation type, whereas the second Brauer-Thrall conjecture
states that the algebras of unbounded
representation type are of strongly unbounded representation type.
Here, we say an algebra is {\it of bounded representation type} if the dimensions of all
indecomposable modules have a common upper bound, and {\it
of strongly unbounded representation type} if there are infinitely
many $d \in \mathbb{N}$ such that for each $d$, there exist infinitely many
isomorphism classes of indecomposable modules of dimension $d$.
 The study of the Brauer-Thrall conjectures, to a large extent,
 stimulated the development of representation theory
 \cite{Aus74, Bau85, NR75, Ro68, Ro78}.

During the last years, the bounded derived categories of algebras have
been studying extensively and play an important role in representation theory
of finite-dimensional algebras.
By a theorem from \cite{Hap88}, there is a full embedding from the
bounded derived category of a finite-dimensional algebra
to the stable module category over its repetitive algebra, which is
an equivalence if and only if the its global dimension is finite.
The theorem bridged together the bounded derived category and the
module category, and hence provided a method
to explore the property of bounded derived category of algebras in terms of
their repetitive algebra, like the
derived representation type \cite{dPena98, GK02}.
Moreover, the classification and distribution of indecomposable objects in the
bounded derived category of an algebra are still important themes
in representation theory of algebras. In this context, the definitive work was
due to Vossieck \cite{Vo01}. He introduced and classified {\it derived discrete algebras}, i.e., the
algebras whose bounded derived categories admit only finitely many
isomorphism classes of indecomposable objects of arbitrarily given
cohomology dimension vector, and proved an algebra is derived discrete if and
only if its repetitive algebra is discrete. Bautista \cite{Bau06} generalized the
definition of derived discrete for the artin algebras.
Motivated by Vossieck's work, Han and Zhang introduced the cohomological range of a
bounded complex, which leads to the concept of strongly derived unbounded algebras
naturally. We say an algebra is {\it strongly derived unbounded} if
there are infinitely
many $r \in \mathbb{N}$ such that for each $r$, there exist infinitely many
isomorphism classes of indecomposable object of cohomological range $r$ in
its bounded derived category.  Moreover, the authors proved
an algebra is either derived discrete or strongly derived unbounded \cite{HZ13}.

During the research of bounded derived category of algebras,
a high emphasis has been placed another category, i.e.,
the category of all minimal complexes
of finitely generated projective modules with degree
concentrated in $[0, m]$, for any fixed integer $m\geq 0$,
and we denote it by $C_m(\proj A)$. Bautista, Souto Salorio and Zuazua
described the AR-triangles in $C_m(\proj A)$, and also observed
their relation with the
AR-triangles in $K^{-,b}(\proj A)$, the homotopy category of 
all right bounded projective complexes with bounded cohomology \cite{BSZ05}.
Moreover in \cite{Bau06}, Bautista established that,
if $k$ is infinite, then a finite-dimensional $k$-algebra is derived discrete if and only if
for any integer $m$, the category $C_m(\proj A)$ does not contain generic objects.
For the representation type, Bautista defined the finite, tame and wild
representation type for $C_m(\proj A)$, and then proved that
$C_m(\proj A)$ is either of tame representation type or of wild representation
type \cite{Bau07}. Futhermore, $A$ is derived discrete if and only if $C_m(\proj A)$ is
of finite representation type for all $m$. In present paper, we first define the strongly
unboundedness of the category $C_m(\proj A)$ for any fixed integer $m$,
and study the strongly unbounded algebras in terms of the associated category
$C_m(\proj A)$ and the representation type of repetitive algebras.
We prove the following

\medskip

{\bf Theorem.} {\it Let $A$ be a finite-dimensional algebra. Then the following
statements are equivalent

{\rm(1)} $A$ is strongly derived unbounded;

{\rm(2)} There exists an integer
$m\geq 1$, such that the category $C_m(\proj A)$ is of strongly unbounded type.

{\rm(3)}  $K^b(\proj A)$ is of strongly unbounded type;

{\rm(4)} The repetitive algebra $\hat{A}$ is of strongly unbounded representation type .}

\medskip

Consider the dichotomy theorem from \cite{HZ13}, we know any algebra $A$ is
derive discrete or strongly derived unbounded. Combined with the equivalent
characterizations of derived discrete algebras from \cite{Bau07, Vo01}, we
can establish the dichotomy on the representation type of $C_m(\proj A)$,
the homotopy category $K^b(\proj A)$ and the repetitive algebra $\hat{A}$
as a corollary.

\medskip

{\bf Corollary} {\it Let $A$ be an algebra. Then we have

{\rm (1)} $C_m(\proj A)$ is of finite representation type for any 
$m$, or there exists an integer
$m'\geq 1$, such that $C_{m'}(\proj A)$ is of strongly unbounded type.

{\rm (2)} $K^b(\proj A)$ is either discrete or of strongly unbounded type;

{\rm (3)} The repetitive algebra $\hat{A}$ is either of discrete representation type
or strongly unbounded representation type.  }

 \medskip

The present paper is organized as follows. In the first section, we
define the strongly unboundedness of $C_m(\proj A)$ and prove
some basic lemmas. In section 2, we observe the
strongly unboundedness of $C_m(\proj A)$
under the derived equivalences and cleaving functors. Moreover,
we study $C_m(\proj A)$ for
representation-infinite algebras, simply connected algebras and
finally prove the main theorem.

\medskip
\noindent{\bf Acknowledgements } The author would like to thank Yang Han and Henning Krause,
for their helpful discussions on this topic.

\section{The strongly unboundedness of $C_m(\proj A)$}

\subsection{Notations and definitions}

Let $A$ be an algebra,  and $\mod A$ be the category of all
finite-dimensional right $A$-modules and $\proj A$ be its full
subcategory consisting of all finitely generated projective right
$A$-modules. Denote by $C(A)$ the category of all complexes of
finite-dimensional right $A$-modules, and by $C^b(A)$ and
$C^{-,b}(A)$ its full subcategories consisting of all bounded
complexes and right bounded complexes with bounded cohomology
respectively. Denote by $C^b(\proj A)$ and $C^{-,b}(\proj A)$ the
full subcategories of $C^b(A)$ and $C^{-,b}(A)$ respectively
consisting of all complexes of finitely generated projective
modules. Denote by $K(A)$, $K^b(\proj A)$ and $K^{-,b}(\proj A)$ the
homotopy categories of $C(A)$, $C^b(\proj A)$ and $C^{-,b}(\proj A)$
respectively. Moreover, $D^b(A)$ is the bounded derived category of
$\mod A$.

From \cite{HZ13}, for any complex $X^{\bullet} \in
D^b(A)$, the
 {\it cohomological length}  is $$\hl(X^{\bullet}) := \max\{\dim H^i(X^{\bullet}) \; | \;
i \in \mathbb{Z}\}, $$
the {\it cohomological width} of $X^{\bullet}$
is $$\hw(X^{\bullet}) := \max\{j-i+1 \; | \; H^i(X^{\bullet}) \neq 0 \neq H^j(X^{\bullet})\},$$
and the {\it cohomological range} of $X^{\bullet}$
is $$\hr(X^{\bullet}) := \hl(X^{\bullet}) \cdot \hw(X^{\bullet}).$$
Note that these numerical invariants preserve under shifts and isomorphisms. Moreover,
the dimension of an $A$-module $M$ is equal to the
cohomological range of the stalk complex with $M$ in degree $0$.

\begin{definition}
{\cite[Def.5]{HZ13}  An algebra $A$ is said to be
{\it strongly derived unbounded} or {\it of strongly derived unbounded type} if
there is an increasing sequence $\{r_i \; | \; i \in \mathbb{N}\}
\subseteq \mathbb{N}$ such that for each $r_i$, up to shifts and
isomorphisms, there are infinitely many indecomposable objects in
$D^b(A)$ of cohomological range $r_i$.  } \end{definition}

Recall that a complex $X^{\bullet}=(X^i, d^i) \in C^b(A)$ is said to
be {\it minimal} if $\Image d^i \subseteq \rad X^{i+1}$ for all $i \in
\mathbb{Z}$, and the {\it width} of $X^{\bullet}$
is $$w(X^{\bullet}) := \max\{j-i+1 \; | \; X^j \neq 0 \neq X^i\}.$$
For any integer $m\geq 0$,
$C_m(\proj A)$ is the subcategory of $C^b(\proj A)$ consisting of all minimal complexes
$P^{\bullet}=(P^i, d^i)$ such that $P^i=0$ for any
$i\notin\{0, 1, \cdots, m\}$.  Following \cite{Bau06, Bau07},
for $P^{\bullet}\in C_m(\proj A)$,
we put the {\it dimension} of $P^{\bullet}$ is
$$\dim(P^{\bullet})=\sum_{i=0}^m \dim P^i.$$ Now we shall define
the strongly unboundedness of $C_m(\proj A)$.

\begin{definition}
{\rm Let $A$ be an algebra and $m\geq 1$ be an integer.
The category $C_m(\proj A)$ is said to be {\it strongly unbounded} or
{\it of strongly unbounded type} if
there is an increasing sequence $\{d_i \; | \; i \in \mathbb{N}\}
\subseteq \mathbb{N}$ such that for each $d_i$, up to
isomorphisms, there are infinitely many indecomposable objects in
$C_m(\proj A)$ of dimension $d_i$.} \end{definition}

\begin{remark}
Since for any algebra $A$ and fixed integer $m$, there is a full embedding
from the category $C_m(\proj A)$ to  $C_{m+1}(\proj A)$, the strongly unboundedness
of $C_m(\proj A)$ implies that the category $C_{m+1}(\proj A)$ is of strongly
unbounded type. In particular, the statement
$C_m(\proj A)$ is strongly unbounded for some integer $m$ is equivalent to that
$C_m(\proj A)$ is of strongly unbounded type for all but finitely many $m$.
\end{remark}

We need two lemmas in the following.

\begin{lemma}
\label{lemma-dim-control}{\rm(See \cite[Lemma 2.2]{Bau06})} Let $A$ be an algebra with
$\dim A=d$ and $P^{\bullet}\in C_m(\proj A)$ such that $\hl(P^{\bullet})=c$. Then
for any $i\in [0, m]$, we have
$$\dim P^i\leq c(d+d^2+\cdots +d^{m-i+1}).$$
\end{lemma}

\begin{proof}
Since $P^{\bullet}=(P^i, d^i)$ is a minimal complex, i.e.,
$\Image d^i \subseteq \rad P^{i+1}$, then for any $i\in [0,m]$,
$$\begin{aligned} \dim P^i
\leq& \dim A\cdot\dim (P^i/\rad P^i)\\
\leq& \dim A\cdot\dim (P^i/\Image d^{i-1}) \\
= &  \dim A\cdot\big(\dim (P^i/\Ker d^i)+ \dim (\Ker d^i/\Image d^{i-1})\big)\\
= &  \dim A\cdot\big(\dim \Image d^{i}+ \dim H^i(P^{\bullet})\big)\\
\leq &\dim A\cdot\big(\dim P^{i+1}+ \dim H^i(P^{\bullet})\big)\\
\leq &d\cdot\big(\dim P^{i+1}+ c\big).
\end{aligned}$$
Then we can get the inequality as required recursively.
\end{proof}

\begin{lemma}
\label{lemma-iso}
Let $A$ be an algebra and $m\geq 0$ be an integer.
Suppose $P^{\bullet}, Q^{\bullet}$ are two objects in $C_m(\proj A)$.
Then

{\rm (1)} $P^{\bullet}$ is indecomposable in $C_m(\proj A)$ if and
only if it is indecomposable as an object in $D^b(A)$.

{\rm (2)} $P^{\bullet}\cong Q^{\bullet}$ in $C_m(\proj A)$ if and
only if $P^{\bullet}\cong Q^{\bullet}$ as objects in $D^b(A)$.
\end{lemma}

\begin{proof}
(1) Since $D^b(A) \simeq K^{-,b}(\proj A)$, which is Krull-Schmidt,
the complex $P^{\bullet}$ is indecomposable in $D^b(A)$ if and only if
it is an indecomposable complex in $K^b(\proj A)$,
and if and only if its endomorphism algebra $\End_{K(A)}(P^{\bullet})$ is
a local algebra. Moreover,
since the complex $P^{\bullet}$ is minimal, all null homotopic cochain maps in
$\End_{C(A)}(P^{\bullet})$ are in $\rad
\End_{C(A)}(P^{\bullet})$. Thus
$$\End_{K(A)}(P^{\bullet}) / \rad
\End_{K(A)}(P^{\bullet}) \linebreak \cong
\End_{C(A)}(P^{\bullet})/\rad
\End_{C(A)}(P^{\bullet}),$$ which implies
$P^{\bullet}$ is indecomposable in $K^b(\proj A)$ if and only if it is
indecomposable in $C_m(\proj A)$.

(2) If $P^{\bullet}\cong Q^{\bullet}$ in $C_m(\proj A)$, then they are
isomorphic in $D^b(A)$. Conversely, suppose
$P^{\bullet}\cong Q^{\bullet}$ in $D^b(A)$ and there is a quasi-isomorphism
$f^{\bullet}: P^{\bullet}\rightarrow Q^{\bullet}$. Then
we have a triangle in $K(A)$
$$P^{\bullet}\stackrel{f^{\bullet}}{\rightarrow} Q^{\bullet}
\rightarrow L^{\bullet}\rightarrow P^{\bullet}[1]$$
such that $L^{\bullet}$ is an acyclic complex. Applying
$\Hom_{K(A)}(Q^{\bullet}, -)$ to the triangle, we have an isomorphism
$\Hom_{K(A)}(Q^{\bullet}, P^{\bullet})\cong \Hom_{K(A)}(Q^{\bullet}, Q^{\bullet})$
induced by $f^{\bullet}$ since $\Hom_{K(A)}(Q^{\bullet}, L^{\bullet})=0$, which
implies $f^{\bullet}$ is a split epimorphism in $K(A)$.
Note that $P^{\bullet}$ and $Q^{\bullet}$ are quasi-isomorphic. 
Then $f^{\bullet}$ is a chain homotopy equivalence, i.e.,
there is a morphism $g^{\bullet}$ such that $1-g^{\bullet}f^{\bullet}$ and
$1-f^{\bullet}g^{\bullet}$ are null homotopic. Since $P^{\bullet}, Q^{\bullet}$
are minimal, all null-homotopic chain maps are nilpotent. Thus
$f^{\bullet}$ and $g^{\bullet}$ are split monomorphisms in $C_m(\proj A)$.
Therefore, $P^{\bullet}\cong Q^{\bullet}$ in $C_m(\proj A)$.
\end{proof}

The following lemma implies the strongly unboundedness of $C_m(\proj A)$
can be defined in terms of the cohomological range as well.

\begin{lemma}\label{lemma-dim-hr}
Let $A$ be an algebra and $m\geq 1$ be an integer.
The category $C_m(\proj A)$ is strongly unbounded
if and only if there is an increasing sequence $\{r_i \; | \; i \in \mathbb{N}\}
\subseteq \mathbb{N}$ such that for each $r_i$, up to
isomorphisms, there are infinitely many indecomposable objects in
$C_m(\proj A)$ of cohomological range $r_i$.
\end{lemma}

\begin{proof} Suppose there is an increasing sequence $\{r_i \; | \; i \in \mathbb{N}\}
\subseteq \mathbb{N}$ and pairwise non-isomorphic objects
$\{P^{\bullet}_{ij}\; | \; i, j\in \mathbb{N} \}$ in
$C_m(\proj A)$ such that $\hr(P^{\bullet}_{ij})=r_i$.
Note that for any object
$P^{\bullet}\in C_m(\proj A)$, $\hr(P^{\bullet})\leq (m+1)\cdot
\dim(P^{\bullet})$.  Moreover by Lemma \ref{lemma-dim-control},
$\dim(P^{\bullet})\leq \hr(P^{\bullet})\cdot (m+1)\cdot (d+d^2+\cdots +d^{m+1})$. Set
$N=(m+1)\cdot (d+d^2+\cdots +d^{m+1})$, then for any $i, j\in \mathbb{N}$,
we have
$$\frac{1}{m+1} \cdot \hr(P_{ij}^{\bullet}) \leq \dim(P_{ij}^{\bullet})
\leq N \cdot \hr(P_{ij}^{\bullet}).$$
In order to show $C_m(\proj A)$ is of strongly unbounded type, we shall
find inductively an increasing sequence
$\{d_i \; | \; i \in \mathbb{N}\} \subseteq \mathbb{N}$ and
infinitely many indecomposable objects $\{Q_{ij}^{\bullet}\in C_m(\proj A)
\; | \; i, j \in \mathbb{N}\}$ which are pairwise different up to
isomorphism such that $\dim(Q_{ij}^{\bullet}) = d_i$ for
all $j \in \mathbb{N}$. For $i=1$,
$0<\dim(P^{\bullet}_{1j})\leq Nr_1$. Then there is $0<d_1\leq Nr_1$ and
infinitely many indecomposable objects $\{Q^{\bullet}_{1j} \; | \; j \in
\mathbb{N}\} \subseteq \{P_{1j}^{\bullet} \; | \; j \in
\mathbb{N}\} $ of dimension $d_1$. Assume that we have found $d_i$. We choose some $r_l$
with $r_l > (m+1) \cdot d_i$. Since
$$d_i < \frac{1}{m+1} \cdot r_l = \frac{1}{m+1} \cdot \hr(X_{lj}^{\bullet})
\leq \dim(P_{lj}^{\bullet}) \leq N \cdot \hr(X_{lj}^{\bullet}) = N
\cdot r_l,$$ we can choose $d_i < d_{i+1} \leq N \cdot r_l$ and
infinitely many indecomposable objects $\{Q^{\bullet}_{i+1,j} \; |
\; j \in \mathbb{N}\} \subseteq \{P_{lj}^{\bullet} \; | \; j \in
\mathbb{N}\}$ which are pairwise non-isomorphism
such that $\dim(Q^{\bullet}_{i+1,j}) = d_{i+1}$ for all
$j \in \mathbb{N}$.

Conversely we suppose $C_m(\proj A)$ is of strongly unbounded. Then we
can construct an increasing sequence $\{r_i \; | \; i \in \mathbb{N}\}
\subseteq \mathbb{N}$ and pairwise non-isomorphic objects
$\{Q^{\bullet}_{ij}\; | \; i, j\in \mathbb{N} \}$ such that
$\hr(Q^{\bullet}_{ij})=r_i$ in the similar way by the inequality
$$\frac{1}{N} \cdot \dim(P^{\bullet}) \leq \hr(P^{\bullet})
\leq (m+1) \cdot \dim(P^{\bullet}),$$ for any $P^{\bullet}\in C_m(\proj A).$
\end{proof}

\section{The proof of Theorem}

\subsection{Simply connected algebras}

Simply connected algebras play an important role in the
representation theory of algebras since any
representation-finite algebra can be transformed
to a simply connected algebra using covering
technique.  We first recall the definition of simply connected algebras
from \cite{AS88}. Fix a connected quiver $(Q, I)$ with $I$ admissible.
For any $\alpha\in Q_1$,
we write its formal inverse $\alpha^{-1}$ with source $s(\alpha^{-1})=t(\alpha)$ and
$t(\alpha^{-1})=s(\alpha)$. A walk in $Q$ is a path $w=w_1w_2\cdots w_n$ with
$w_i\in Q_1$ or $w_i^{-1}\in Q_1$ such that $s(w_{i+1})=t(w_i)$. An relation
$r=\sum_{i=1}^m t_iu_i\in I(m\geq 1)$ with $u_i$ pairwise distinct and $t_i\in k\setminus \{0\}$
is called {\it minimal} if $r=\sum_{i\in S} t_iu_i\notin I$ for any non-empty proper
subset $S\subset \{1, 2, \cdots, m\}$. The {\it homotopy relation} is the smallest
equivalence relation $\sim_I$ on the set of walks such that

(1) $\alpha\alpha^{-1}\sim_I e_x$ and $\alpha^{-1}\alpha\sim_I e_y$ for any
$x\stackrel{\alpha}{\rightarrow}y$;

(2) $u_1\sim_I u_2$ for any minimal relation $t_1u_1+t_2u_2+\cdots +t_mu_m$;

(3) $u\sim_I v$ implies $uw\sim_I vw$ and $wu\sim_I wv$ for any $w$.

\noindent The {\it fundamental group} $\Pi_1(Q, I, x_0)$ of $(Q, I)$ is defined to the group
consisting of homotopy classes of walks
from $x_0$ to $x_0$ for any vertex $x_0\in Q_0$ \cite{DM83}. Note that the definition
is independent of the choice of $x_0$, and we write $\Pi_1(Q, I)$ for short.
A triangular algebra $A$ is said to be {\it simply connected}
if for any presentation $A\cong kQ/I$, the fundamental group
$\Pi_1(Q, I)$ is trivial.

\medskip

The following lemma implies that
for a representation-infinite
algebra $A$, the category $C_1(\proj A)$
is of strongly unbounded type.

\begin{lemma}\label{lemma-rep-inf}
Let $A$ be a representation-infinite algebra. Then $C_1(\proj A)$ is of
strongly unbounded type.
\end{lemma}

\begin{proof}
Since $A$ is representation-finite, $A$ is of strongly unbounded type, i.e.,
there is an infinite sequence
$\{d_i \; | \; i \in \mathbb{N}\} \subseteq \mathbb{N}$ and
infinitely many indecomposable $A$-modules $\{M_{ij}
\; | \; i, j \in \mathbb{N}\}$ which are pairwise different up to
isomorphism such that $\dim(M_{ij}) = d_i$ for all $j \in \mathbb{N}$.
For any $M_{ij}$, we can take a minimal presentation
$P^{-1}\stackrel{d}{\rightarrow} P^0\rightarrow M_{ij}\rightarrow 0$. Let
$$P_{ij}^{\bullet}=\cdots\rightarrow 0\rightarrow P^{-1}\stackrel{d}{\rightarrow}
P^0\rightarrow 0\rightarrow \cdots$$
with $P^{-1}$ concentrated in degree 0. Then $P_{ij}^{\bullet}\in C_1(\proj A)$
is indecomposable by \cite[Prop.2]{HZ13} with $\dim H^1(P_{ij}^{\bullet})=d_i$.
Moreover, $P_{ij}^{\bullet}$
are non-isomorphic for different $i, j\in\mathbb{N}$.
Since $P_{ij}^{\bullet}$ is a minimal presentation of $M_{ij}$,
$\dim P^{-1}\leq (\dim A)^2\cdot d_i$ and we have
$$d_i\leq \hr(P_{ij}^{\bullet})\leq 2\cdot (\dim A)^2\cdot d_i.$$
With the similar argument in the proof of Lemma \ref{lemma-dim-hr},
we can construct a sequence $\{r_i \; | \; i \in \mathbb{N}\}
\subseteq \mathbb{N}$ and pairwise non-isomorphic objects
$\{Q^{\bullet}_{ij}\; | \; i, j\in \mathbb{N} \}$ such that
$\hr(Q^{\bullet}_{ij})=r_i$. Thus $C_1(\proj A)$ is of strongly unbounded
type by Lemma \ref{lemma-dim-hr}.
\end{proof}

The following lemma observe the strongly unboundedness of $C_m(\proj A)$ under the
derived equivalences.

\begin{proposition}\label{prop-der-equiv}
Let $A$ be an algebra with $C_m(\proj A)$ strongly unbounded for some
integer $m$ and $\mbox{\rm gl.dim} A<\infty$. If there is an
algebra $B$ derived equivalent to $A$, then
$C_{m'}(\proj B)$ is of strongly unbounded type for some integer
$m'$.
\end{proposition}

\begin{proof}
Since $C_m(\proj A)$ is strongly unbounded, by Lemma \ref{lemma-dim-hr},
there is an increasing sequence $\{r_i \; | \; i \in \mathbb{N}\}
\subseteq \mathbb{N}$ and pairwise non-isomorphic objects
$\{P^{\bullet}_{ij}\; | \; i, j\in \mathbb{N} \}$ in
$C_m(\proj A)$ such that $\hr(P^{\bullet}_{ij})=r_i$.
Moreover, since $A$ and $B$ are derived equivalent,
there is a two-sided tilting complex in $D^b(A^{op}\otimes B)$
$$_AT^{\bullet}_B= 0\rightarrow T^{-l} \rightarrow
T^{-l+1} \rightarrow \cdots \rightarrow
T^{-1}\rightarrow P^0 \rightarrow 0, $$ such that $F=-
\otimes^L_A T^{\bullet}_B : D^b(A) \rightarrow D^b(B)$ is a derived
equivalence \cite{Ri91}. Note that $\mbox{\rm gl.dim} A<\infty$ implies
$\mbox{\rm gl.dim} B<\infty$ \cite{Hap88}. We assume
$\mbox{\rm gl.dim} B=n$, and we can take a minimal projective $B$-$B$-bimodule
resolution of $B$
$$R^{\bullet} = 0 \longrightarrow R^{-n} \stackrel{d^{-n}}{\longrightarrow}
R^{-n+1} \stackrel{d^{-n+1}}{\longrightarrow} \cdots \longrightarrow
R^{-1} \stackrel{d^{-1}}{\longrightarrow} R^0 \longrightarrow 0.$$
Then for any $i, j\in \mathbb{N}$, $F(P_{ij}^{\bullet})
=P_{ij}^{\bullet}\otimes^L_A T^{\bullet}_B\cong
P_{ij}^{\bullet}\otimes^L_A T^{\bullet}\otimes^L_B R^{\bullet}_B$,
which is a projective $B$-module complex of width less than
$m+l+n$. Thus, without loss of generality, we can assume
$F(P_{ij}^{\bullet})\in C_{m+l+n}(\proj B)$ with suitable shifts and isomorphisms
for any $i, j\in \mathbb{N}$. By \cite[Prop.1(3)]{HZ13}, we have two
integers $N, N'$, such that
$$\frac{1}{N'} \cdot \hr(P_{ij}^{\bullet}) \leq
\hr(F(P_{ij}^{\bullet})) \leq N \cdot \hr(P_{ij}^{\bullet}).$$
With a similar discussion as in the proof of Lemma \ref{lemma-dim-hr},
we shall find inductively an increasing sequence
$\{r'_s \; | \; s \in \mathbb{N}\}$ and
infinitely many indecomposable pairwise non-isomorphic
objects $\{Q_{st}^{\bullet}\in C_{m+l+n}(\proj B)
\; | \; s, t \in \mathbb{N}\}\subseteq \{F(P_{ij}^{\bullet})\; |\; i, j\in
\mathbb{N}\}$ such that $\hr(Q_{st}^{\bullet}) = r'_s$.
Thus the lemma follows by Lemma \ref{lemma-dim-hr}.
\end{proof}

\begin{corollary}
\label{coro-simply-con}
Let $A$ be a simply connected algebra. If $A$ is strongly derive unbounded,
then there exists an integer $m$ such that $C_m(\proj A)$ is of
strongly unbounded type.
\end{corollary}

\begin{proof}
By \cite[Lemma 2]{HZ13}, any simply connected algebra is tilting equivalent to
a hereditary algebra of Dynkin type or a representation-infinite algebra.
If $A$ is strongly derived unbounded, then $A$ is tilting equivalent to
a representation-infinite algebra. Since simply connected algebras are triangular
algebras and then of finite global dimension,
by the previous proposition and Lemma \ref{lemma-rep-inf},
there exists an integer $m$ such that $C_m(\proj A)$ is of
strongly unbounded type.
\end{proof}

\subsection{Cleaving functors and the strongly unboundedness of $C_m(\proj A)$}

In the context of cleaving functors, bound quiver algebras are
viewed as bounded categories, see \cite{GR97} for details.
In the rest of this paper, we will replace bound quiver algebras by bounded categories.

A $k$-linear category $A$ is a category together with $k$-vector space structure
on the set $A(x, y)$ of all morphisms from $x\in A$ to $y\in A$ such that the
composition of morphisms is bilinear.
We say a $k$-linear category $A$ is a {\it locally bounded category} if

(1) different objects in $A$ are non-isomorphic;

(2) for any $a\in A$, the endomorphism algebra $A(a, a)$ is local;

(3) $\dim_k \sum_{x \in A} A(a,x) < \infty$ and $\dim_k \sum_{x \in
A} A(x, a) < \infty$ for all $a \in A$.

\noindent A locally bounded category is a
{\it bounded category} if it has only finitely many objects.
Note that a bound quiver algebra $A=kQ/I$ with
$I$ admissible can be viewed as a bounded category
by seeing the vertexes $i\in Q_0$ as objects and the
combinations of paths in $kQ/I$ as morphisms.
Conversely, a bounded category $A$ admits a presentation
$A\cong kQ_A/I$ with $Q_A$  finite and
$I$ admissible.

Let $A$ be a locally bounded category. A {\it right $A$-module} $M$ is just a
covariant $k$-linear functor from $A$ to the category of $k$-vector
spaces.  Denote by $\Mod A$ the category of all right
$A$-modules $M$ with $\dim M(a)<\infty$ for any $a\in A$.
For any $M\in \Mod A$, the {\it dimension vector} of $M$ is
$\mathbf{dim} M:=(\dim M(a))_{a\in A}$, and the {\it
support} of $M$ is $\Supp M:=\{a\in A \; |\; M(a)\neq 0\}$.
Denote by $\mod A$ the full subcategory of $\Mod A$
consisting of all $A$-modules $M$ such that
$\Supp M$ is finite. The {\it dimension} of $M\in \mod A$ is $\dim M := \sum_{a \in
A}\dim_k M(a)$. The indecomposable
projective $A$-modules are $P_a=A(a, -)$ and indecomposable
injective $A$-modules are $I_a=DA(-, a)$ for all $a \in A$, where
$D=\Hom_k(-, k)$. Moreover, all the concepts and notations defined
for a bound quiver algebra make sense for a bounded
category.

To a $k$-linear functor $F: B \rightarrow A$ between bounded
categories, we associates a {\it restriction functor} $F_{\ast}:
\mod A \rightarrow \mod B$, which is given by $F_{\ast}(M) = M \circ
F$ and exact. The restriction functor $F_{\ast}$ admits a left
adjoint functor $F^{\ast}$, called the {\it extension functor},
which sends a projective $B$-module $B(b, -)$ to a projective
$A$-module $A(Fb, -)$. Moreover, $F_{\ast}$ extends naturally to a
derived functor $F_{\ast}: D^b(A)\rightarrow D^b(B)$, which has a
left adjoint $\mathbf{L}F^{\ast}: D^b(B)\rightarrow D^b(A)$. Note
that $\mathbf{L}F^{\ast}$ is the left derived functor associated
with $F^{\ast}$ and maps $K^b(\proj B)$ into $K^b(\proj A)$. We
refer to \cite{W94} for the definition of derived  functors.

A $k$-linear functor $F: B \rightarrow A$ between bounded categories
is called a {\it cleaving functor} \cite{BGRS85,Vo01} if it
satisfies the following equivalent conditions:

(1) The linear map $B(b,b') \rightarrow A(Fb,Fb')$ associated with
$F$ admits a natural retraction for all $b,b' \in B$;

(2) The adjunction morphism $\phi_M: M \rightarrow (F_{\ast} \circ
F^{\ast})(M)$ admits a natural retraction for all $M \in \mod B$;

(3) The adjunction morphism $\Phi_{X^{\bullet}}: X^{\bullet}
\rightarrow (F_{\ast} \circ \mathbf{L}F^{\ast})(X^{\bullet})$ admits
a natural retraction for all $X^{\bullet} \in D^b(B)$.

\begin{proposition}\label{prop-cleaving}
Let $B$ be a bounded category of finite global dimension and
$C_m(\proj B)$ be of strongly unbounded type for some $m$. If
there is a cleaving functor $F: B\rightarrow A$,
then $C_{m}(\proj A)$ is of strongly unbounded type.
\end{proposition}

\begin{proof}
Suppose there is an increasing sequence $\{r_i \; | \; i \in \mathbb{N}\}
\subseteq \mathbb{N}$ and pairwise non-isomorphic objects
$\{P^{\bullet}_{ij}\; | \; i, j\in \mathbb{N} \}$ in
$C_m(\proj B)$ such that $\hr(P^{\bullet}_{ij})=r_i$.
Since $F$ is a cleaving functor, for any $i,j\in \mathbb{N}$,
$\mathbf{L}F^{\ast}(P_{ij}^{\bullet})=F^{\ast}(P_{ij}^{\bullet})$,
which is projective $A$-module complex of width less 
than $m$ by the definition of $F^{\ast}$. Then, with suitable 
isomorphisms, we can assume 
$\mathbf{L}F^{\ast}(P_{ij}^{\bullet})$ lies in $C_m(\proj A)$.
Moreover, for any $i, j\in\mathbb{N}$, $P_{ij}^{\bullet}$ is
a direct summand of $(F_{\ast} \circ \mathbf{L}F^{\ast})
(P_{ij}^{\bullet})$. Thus for any $P_{ij}^{\bullet}$, we can
choose an indecomposable direct summand $Q_{ij}^{\bullet}$ of
$\mathbf{L}F^{\ast}(P_{ij}^{\bullet})$, such that $P_{ij}^{
\bullet}$ is a direct summand of $F_{\ast}(Q_{ij}^{\bullet})$.
Note that for any $i \in \mathbb{N}$, the set $\{Q_{ij}^{\bullet} \;
| \; j \in \mathbb{N}\}$ contains infinitely many elements which are
pairwise non-isomorphic since the set $\{P_{ij}^{\bullet} \;
| \; j \in \mathbb{N}\}$ contains infinitely many pairwise non-isomorphic
elements. Moreover, by the proof of \cite[Prop.5(1)]{HZ13},
there exist two integers $N, N'$, such that for any $i,j \in \mathbb{N}$,
we have the inequality $\frac{1}{N'} \cdot \hr(P_{ij}^{\bullet})
\leq \hr(Q_{ij}^{\bullet}) \leq N \cdot \hr(P_{ij}^{\bullet})$.
Thus $C_m(\proj A)$ is of strongly unbounded with a similar discussion
as in the Lemma \ref{lemma-dim-hr}.
\end{proof}

\subsection{The proof of the main theorem}

Let $A$ be a bounded category. Recall that the repetitive category $\hat{A}$
of $A$ has the pairs $(a, i)$ as objects, where $a\in A$ and $i\in \mathbb{Z}$,
while the morphisms from $(a, i)$ to $(b, i)$ and $(b, i+1)$ are determined by
$A(a, b)$ and $A(b, a)$ respectively, and zero else \cite{HW83}. Note that $\hat{A}$ is self-injective
locally bounded category. Moreover, there is a full embedding
triangulated functor $F: D^b(A)\rightarrow \underline{\mod} \hat{A}$ \cite{Hap88}.

Recall from \cite{Vo01}, $A$ is said to be {\it derived discrete} if for any
$d\in \mathbb{N}$, there are only finitely many indecomposables in $D^b(A)$
with cohomological range $d$. Moreover, $K^b(\proj A)$ is {\it discrete} if
for any
$d\in \mathbb{N}$, there are only finitely many indecomposables in $K^b(\proj A)$
of cohomological range $d$.

\begin{definition}
A locally bounded category $B$ is said to be {\it of discrete representation type } if
for any $\mathbf{d}\in \mathbb{N}^{|B|}$, there are only finitely many indecomposable
objects $M \in \mod A$ with $\mathbf{dim} M=\mathbf{d}$. Moreover, we say
$B$ is {\it of strongly unbounded representation type} if there are infinitely many
$\mathbf{d}\in \mathbb{N}^{|B|}$ such that for each $\mathbf{d}$, there are
infinitely many indecomposables in $\mod A$ with dimension vector $\mathbf{d}$.
\end{definition}

The following lemma is the classification of derived discrete algebras
due to Vossieck \cite[Theorem]{Vo01}.

\begin{lemma}\label{lemma-voss}
Let $A$ be a bounded category. Then the following statements are equivalent

{\rm(1)} $\hat{A}$ is of discrete representation type;

{\rm(2)}  $A$ is derived discrete;

{\rm(3)}  $K^b(\proj A)$ is discrete;

{\rm(4)}  $A$ is piecewise hereditary of Dynkin type or admits a presentation $kQ/I$
with $Q$ one-cycle gentle quiver such that the numbers of clockwise and of counterclockwise
paths of length two which belongs to $I$ are different.
\end{lemma}

\begin{definition}
Let $A$ be a bounded category. $K^b(\proj A)$ is said to be
{\it of strongly unbounded type} if
there is an increasing sequence $\{r_i \; | \; i \in \mathbb{N}\}
\subseteq \mathbb{N}$ such that for each $r_i$, up to shifts and
isomorphisms, there are infinitely many indecomposable objects in
$K^b(\proj A)$ of cohomological range $r_i$.
\end{definition}

Now we can prove the Theorem.

\begin{theorem}
Let $A$ be a bounded category. Then the following
statements are equivalent

{\rm(1)} $A$ is strongly derived unbounded;

{\rm(2)} There exists an integer
$m\geq 1$, such that the category $C_m(\proj A)$ is of strongly unbounded type.

{\rm(3)}  $K^b(\proj A)$ is of strongly unbounded type;

{\rm(4)}  $\hat{A}$ is of strongly unbounded representation type.
\end{theorem}

\begin{proof}
(1)$\Rightarrow$(2):  We assume for any integer $m>0$,
$C_m(\proj A)$ is not of strongly unbounded type.
Then $A$ is representation-finite by Lemma \ref{lemma-rep-inf}.
Thus for any $a \in A$, $A(a,a)$ is a
uniserial local algebra, and then $A(a,a) \cong k$ or $A(a,a) \cong k[x]/(x^l)$
with $l \geq 2$. Moreover, we can exclude the cases
$A(a,a) \cong k[x]/(x^l)$ for $l \geq 3$. Indeed, we consider the
the functor $F : A_n^l \rightarrow A$
given by $F(i)=a$ and $F(\alpha_j)=x$, where
$A_n^l$ is the bounded category defined by the quiver
$$\xymatrix{
n \ar [r]^{\alpha_{n-1}} & n-1\ar [r]^{\alpha_{n-2}}& \cdots \ar
[r]^{\alpha_2}& 2 \ar [r]^{\alpha_1} &1},$$ and the admissible ideal
generated by all paths of length $l$. Note that $F$ is a
cleaving functor. By the construction in \cite[Lemma 4]{HZ13},
if $l\geq 3$, then $C_5(\proj A_{3l}^l)$ is strongly unbounded,
and thus $C_5(\proj A)$ is of strongly unbounded type by
Proposition \ref{prop-cleaving}, which is a contradiction. Therefore,
for any $a \in A$, $A(a,a)\cong k$ or $A(a,a) \cong k[x]/(x^2)$.
By \cite[Section 9]{BGRS85}, $A$ is standard since $A$ contains
no Riedtmann contours.

If $A$ is simply connected, then $A$ is not strongly unbounded by
Corollary \ref{coro-simply-con}.  Suppose $A$ is not simply connected,
then there is a Galois covering $\pi: \tilde{A}\rightarrow A$ with
non-trivial free Galois group $G$ and $\tilde{A}$ simply
connected \cite{BG83, Ga81}. Now we consider the full convex
subcategory $B$ of $\tilde{A}$. Then $B$ is also simply connected.
Since the composition of the embedding $i: B\hookrightarrow \tilde{A}$
and $\pi$ is cleaving functor, $B$ is not strongly derived unbounded
by Corollary \ref{coro-simply-con}. Thus $B$ is piecewise hereditary
of Dynkin type \cite[Lemma 2]{HZ13}. Then $B$ is piecewise hereditary of
type $\mathbf{A}$ with the same argument as that in the proof of
\cite[Lemma 4.4]{Vo01} and $\tilde{A}$ admits a presentation given by a
gentle quiver $(Q,I)$ (Ref. \cite[Theorem]{AH81}), and so does $A$.
By Bekkert and Merklen's classification on the indecomposable objects in
the derived category of a gentle algebra \cite{BM03},
if $A$ contains a generalized band $w$, then we can construct
a family of pairwise non-isomorphic indecomposables
$P^{\bullet}_{w,f}$ for $f=(x-\lambda)^d\in k[x]$ in $C_m(\proj A)$
for some integer $m$,
such that $P^{\bullet}_{w,f}$ and $P^{\bullet}_{w,f'}$ have the same dimension
if and only if $\deg(f)=\deg(f')$, where $\lambda\in k\setminus\{0\}$ and $d>0$.
Then $C_m(\proj A)$ is of strongly unbounded type, which is a contradiction.
Thus $A$ contains no generalized bands and then
$A$ is derived discrete by \cite[Theorem 4]{BM03}. Therefore,
$A$ is not strongly derived unbounded.

(2)$\Rightarrow$(3):
Suppose there exists an integer
$m\geq 1$, such that $C_m(\proj A)$ is of strongly unbounded type.
Then by Lemma \ref{lemma-dim-hr},
there is an increasing sequence $\{r_i \; | \; i \in \mathbb{N}\}
\subseteq \mathbb{N}$ and pairwise non-isomorphic objects
$\{P^{\bullet}_{ij}\; | \; i, j\in \mathbb{N} \}$ in
$C_m(\proj A)$ such that $\hr(P^{\bullet}_{ij})=r_i$.
Since the elements in $\{P^{\bullet}_{ij}\; | \; i, j\in \mathbb{N} \}$,
 seen as objects in $K^b(\proj A)$,
are also pairwise non-isomorphic indecomposables,
$K^b(\proj A)$ is strongly unbounded.

(3)$\Rightarrow$(1): Trivial.

(2)$\Rightarrow$(4): Suppose $C_m(\proj A)$ is of strongly unbounded type, by
Lemma \ref{lemma-dim-hr},
there is an increasing sequence $\{r_i \; | \; i \in \mathbb{N}\}
\subseteq \mathbb{N}$ and pairwise non-isomorphic complexes
$\{P^{\bullet}_{ij}\; | \; i, j\in \mathbb{N} \}$ in
$C_m(\proj A)$ such that $\hr(P^{\bullet}_{ij})=r_i$.
Note that $\{P^{\bullet}_{ij}\; | \; i, j\in \mathbb{N} \}$
are pairwise non-isomorphic indecomposables viewed as objects
in $D^b(A)$ by Lemma \ref{lemma-iso}. Assume
$\{S_a\;|\; a\in A\}$ and $\{S_h\;|\; h\in \hat{A}\}$ are
the sets of all simple
$A$-modules and $\hat{A}$-modules respectively.

Now we consider the full embedding
$F: D^b(A)\rightarrow \underline{\mod} \hat{A}$.
On one hand, for a fixed object
$X^{\bullet}\in D^b(A)$, without loss of generality, we assume it concentrated
in degree $[0, n]$. Note that $X^{\bullet}$ is
generated by the cohomologies via triangles and
the comhomologies can be also obtained by triangles with the simples.
Since $F$ sends a triangle in $D^b(A)$ to a triangle in $\underline{\mod} \hat{A}$,
by the additivity of dimension functor $\mathbf{dim} (-)$
in $\underline{\mod} \hat{A}$, we have the
following estimate (see also \cite{Vo01})
$$\mathbf{dim} F(X^{\bullet})\leq \sum_{a\in A}\sum_{l=0}^n
\dim H^l(X^{\bullet})(a)\cdot \mathbf{dim} F(S_a[l])
\leq \hr(X^{\bullet})\cdot\sum_{a\in A}\sum_{l=0}^n
\mathbf{dim} F(S_a[l]).$$
Set $\mathbf{d}=\sum_{a\in A}\sum_{l=0}^n
\mathbf{dim} F(S_a[l])$. Then for any $i,j\in\mathbb{N}$, $\mathbf{dim}
F(P_{ij}^{\bullet})\leq r_i\cdot\mathbf{d}.$

On the other hand, for any object
$X^{\bullet}\in D^b(A)$, we have
$$\begin{aligned} \hr(X^{\bullet})
=&\sum_{l\in \mathbb{Z}}\dim H^l(X^{\bullet})
= \sum_{l\in \mathbb{Z}}\dim \Hom_{D^b(A)}(A[l], X^{\bullet})\\
=& \sum_{l\in \mathbb{Z}}\dim \Hom_{\hat{A}}(F(A[l]), F(X^{\bullet})) \\
\leq & \sum_{l\in \mathbb{Z}} \sum_{h\in\hat{A}}c_h(F(A[l]))\dim \Hom_{\hat{A}}(S_h, F(X^{\bullet}))\\
\leq & \sum_{l\in \mathbb{Z}} \sum_{h\in\hat{A}}c_h(F(A[l]))\dim \Hom_{\hat{A}}(P_h, F(X^{\bullet}))\\
\leq &\sum_{l\in \mathbb{Z}}\sum_{h\in\hat{A}}c_h(F(A[l]))\dim F(X^{\bullet})(h),
\end{aligned}$$
where $c_h(F(A[l]))$ denotes the
the number of composition factors of $F(A[l])$ isomorphic to $S_h$, and
$P_h$ is the indecomposable projective $\hat{A}$-module associated to
$h\in \hat{A}$.
Set $c=\sup\{c_h(F(A[l])) \;|\; 0\leq l\leq m,  h\in \hat{A}\}$.
Then for any $i,j\in \mathbb{N}$, 
$\hr(P_{ij}^{\bullet})\leq c\cdot \hw(P_{ij}^{\bullet})\cdot\dim F(P_{ij}^{\bullet})$. 
Thus $\frac{1}{c}\cdot \hl(P_{ij}^{\bullet})\leq \dim F(P_{ij}^{\bullet})$.

To prove $\hat{A}$ is of strongly unbounded representation type, we shall
find inductively infinitely vectors
$\{\mathbf{d_i} \; | \; i \in \mathbb{N}\} $ and
infinitely many indecomposable objects $\{M_{ij}^{\bullet}\in \mod \hat{A}
\; | \; i, j \in \mathbb{N}\}$ which are pairwise different up to
isomorphism such that $\mathbf{dim} M_{ij}^{\bullet} = \mathbf{d_i}$ for
all $j \in \mathbb{N}$. For $i=1$,
we have $0<\mathbf{dim} F(P^{\bullet}_{1j})\leq r_1\cdot \mathbf{d}$.
Then there is $0<\mathbf{d_1}\leq r_1\cdot \mathbf{d}$ and
infinitely many indecomposable objects $\{M_{1j} \; | \; j \in
\mathbb{N}\} \subseteq \{F(P_{1j}^{\bullet}) \; | \; j \in
\mathbb{N}\} $ of dimension vector $\mathbf{d_1}$.
Assume that we have done for $i$, and
$d_i=\sum_{j\in\mathbb{Z}}(\mathbf{d_i})_j$. Then we choose some $r_l$
with $r_l > c(m+1) \cdot d_i$, and thus $\hl(P_{lj}^{\bullet})>c\cdot d_i$. Since
$d_i < \frac{1}{c} \cdot \hl(P_{lj}^{\bullet})
\leq \dim(F(P_{lj}^{\bullet}))$, and
$\mathbf{dim}
F(P_{lj}^{\bullet})\leq r_l\cdot\mathbf{d},$
we can choose a vector $\mathbf{d_{i+1}}$, which is different from
$\{\mathbf{d_s} \; |\; s=1, 2,\cdots, i\}$, such that
$\mathbf{d_{i+1}} \leq  r_l\cdot\mathbf{d}$,  and
infinitely many pairwise non-isomorphism indecomposable objects $\{M_{i+1,j} \; |
\; j \in \mathbb{N}\} \subseteq \{F(P_{lj}^{\bullet}) \; | \; j \in
\mathbb{N}\}$ with $\mathbf{dim} M_{i+1,j} = \mathbf{d_{i+1}}$ for all
$j \in \mathbb{N}$.

(4)$\Rightarrow$(1): Suppose $A$ is not strongly derived unbounded, then by \cite[Theorem 2]{HZ13},
$A$ is derived discrete. Thus $\hat{A}$ is representation
discrete by Lemma \ref{lemma-voss}, which is a contradiction with the
assumption.
\end{proof}

Recall from \cite{Bau07}, for an algebra $A$ and a fixed integer $m$,
the category $C_m(\proj A)$ is said to be
{\it of finite representation type} if $C_m(\proj A)$ contains only finitely
many indecomposables up to isomorphisms. As a corollary of the previous
theorem, we obtain the dichotomy on the representation type of $C_m(\proj A)$,
$K^b(\proj A)$ and also the repetitive algebra $\hat{A}$.

\begin{corollary}
Let $A$ be an algebra. Then we have

{\rm (1)} $C_m(\proj A)$ is of finite representation type for any $m$, or there exists an integer
$m'\geq 1$, such that $C_{m'}(\proj A)$ is of strongly unbounded type.

{\rm (2)} $K^b(\proj A)$ is either discrete or of strongly unbounded type;

{\rm (3)} The repetitive algebra $\hat{A}$ is either of discrete representation type
or strongly unbounded representation type.
\end{corollary}

\begin{proof} By \cite[Theorem 2.4(1)]{Bau07}, we know that
$A$ is derived discrete if and only if any $C_m(\proj A)$ is
of finite representation type. Moreover,
$A$ is strongly derived unbounded if and only if
there exists an integer
$m\geq 1$, such that $C_m(\proj A)$ is of strongly unbounded type
by the previous theorem. Since any algebra $A$ is
either derived discrete or strongly derived unbounded
by \cite[Theorem 2]{HZ13}, the statement (1) follows.
Similarly, the statements (2) and (3) hold by the Lemma \ref{lemma-voss} and the
previous theorem.
\end{proof}

\end{document}